\documentclass[11pt]{article}
\usepackage[english,activeacute]{babel}
\usepackage{amsmath,amsfonts,amssymb,amstext,amsthm,amscd,mathrsfs,amsbsy,xypic,centernot}
\usepackage{xcolor}
\usepackage{hyperref}

\newtheorem{teo}{Theorem}[section]
\newtheorem{pro}[teo]{Proposition}
\newtheorem{coro}[teo]{Corollary}
\newtheorem{lem}[teo]{Lemma}

\theoremstyle{definition}
\newtheorem{defi}[teo]{Definition}
\newtheorem{exam}[teo]{Example}
\newtheorem{rem}[teo]{Remark}

\newcommand{\N}{\mathbb N}
\newcommand{\Z}{\mathbb Z}

\newcommand{\R}{\mathbb R}

\newcommand{\K}{\mathbb K}

\newcommand{\pol}{\mathbb K[x_1,\ldots,x_n]}
\newcommand{\plx}{\mathbb K[x]}
\newcommand{\pls}{\mathbb K[S]}

\newcommand{\az}{\mathfrak{a}_z}
\newcommand{\bz}{\mathfrak{b}_{(z,\lambda)}}

\newcommand{\A}{\mathcal{A}}
\newcommand{\ia}{I_{\mathcal{A}}}

\newcommand{\iaw}{I_{{\mathcal{A}}_w}}
\newcommand{\zz}{\underline{0}}
\newcommand{\pia}{\pi_{\A}}
\newcommand{\piw}{\pi_{\A_w}}
\newcommand{\sat}{\overline{\N\A}}
\newcommand{\stw}{\overline{\N\A_w}}

\newcommand{\sA}{\sim_{\A}}

\newcommand{\LT}{\operatorname{LT}}
\newcommand{\Bet}{\operatorname{Betti}}
\newcommand{\Betm}{\operatorname{Betti-min}}

\newcommand{\igno} [1] {}

\providecommand{\keywords}[1]{{\textbf{Keywords:}} #1}

\begin{document}

\title{A semigroup defining the Gr\"obner degeneration of a toric ideal}

\author{Hern\'an de Alba Casillas\footnote{Research supported by Consejo Nacional de Ciencia y Tecnolog\'ia grant A1-S-30482.}, Daniel Duarte, Ra\'ul Vargas Antuna}

\maketitle

\begin{abstract}
We give an explicit set of generators for the semigroup of the Gr\"obner degeneration of a toric ideal. This set of generators is used to study algebraic properties of the semigroup it generates: approximation of semigroups, non-preservation of saturation, Betti elements, uniqueness of presentations, and M\"obius functions.
\end{abstract}

\tableofcontents

\noindent\keywords{Gr\"obner degeneration, toric ideal, semigroups, generators.}

%\\
%\msc{}

%%%%%%%%%%%%%%%%%%%%%%%%%%%%%%%%%%%%
%%%%%%%%%%%%%%%%%%%%%%%%%%%%%%%%%%%%

\section*{Introduction} 

The Gr\"obner degeneration of an ideal is a flat family that deforms the ideal into a simpler one, for instance, a monomial ideal. Together with the theory of Gr\"obner bases, Gr\"obner degenerations are a powerful tool in the study of algebraic and geometric properties of an ideal or the corresponding algebraic variety. 

In this paper we study Gr\"obner degenerations of toric ideals. The starting point for the whole discussion is the fact that the Gr\"obner degeneration of a toric ideal is also a toric ideal. Hence, it is defined by some semigroup. Several natural questions can be deduced concerning the algebraic or combinatorial properties of this semigroup. Yet, somehow surprisingly, they have not been systematically explored as far as we know. This paper aims at giving a first step towards that study.

Let $S\subset\Z^d$ be a semigroup defining a toric ideal $I$. Let $S_w\subset\Z^{d+1}$ be the semigroup of the Gr\"obner degeneration of $I$ with respect to some weight vector $w$. Our first result provides an explicit set of generators of $S_w$ (Theorem \ref{main}). The resulting set of generators is our main tool to explore to what extent properties of the semigroup $S$ are preserved for $S_w$.

We first study approximations of semigroups. It is known that the semigroup $S$ can be approximated by its saturation $\overline{S}$. This means that there exists $a\in S$ such that $a+\overline{S}\subset S$. We prove that the corresponding result for the semigroup $S_w$ also involves the element $a\in S$ (Theorem \ref{approx}). We stress that approximation of semigroups has a great deal of applications, including the study of valuations on graded algebras, the computation of dimension and degree of projective varieties, and intersection theory \cite{KK}.

Next, we show that saturation is not preserved under Gr\"obner degenerations. We exhibit an infinite family of saturated semigroups for which a degeneration is saturated only for finitely many of them (Section \ref{sat}). This suggests that saturation is actually rarely preserved under degenerations. Even though this phenomenon was somehow predictable, there are no explicit examples illustrating this fact in the literature, as far as we are aware. 

We then study Betti elements of semigroups, a concept that is closely related to syzygies of toric ideals \cite{GR1}. As before, the goal is to study the behaviour of Betti elements under Gr\"obner degenerations. We first show that each Betti element of the semigroup $S$ gives rise to a Betti element of $S_w$ (Theorem \ref{inclusion}). However, in general not all Betti elements of $S_w$ can be described like this (Example \ref{scroll}). 

A particular case that we study is that of semigroups having a unique minimal generating set. These semigroups appear, for instance, in Algebraic Statistics \cite{T}. We show that this property is preserved in the degeneration for certain families (Sections \ref{unique min} and \ref{unique min 2}). In our opinion, this gives enough evidence to conjecture that the uniqueness of a minimal generating set is always preserved under Gr\"obner degenerations (Conjecture \ref{conjdAD}).

The last section concerns M\"obius functions of semigroups. This notion was introduced by G.-C. Rota and has a great deal of applications \cite{R}. For a long time, M\"obius functions were investigated only for numerical semigroups, although more general results were obtained in recent years \cite{CGMR}. In particular, an explicit formula for the M\"obius function of semigroups with unique Betti element was given in \cite[Theorem 4.1]{CGMR}. As an application of this formula, we conclude this paper by showing that the M\"obius function of $S_w$ can be computed using only data of $S$ and $w$, whenever both $S$ and $S_w$ have a unique Betti element (Theorem \ref{mu hiper}).

%%%%%%%%%%%%%%%%%%%%%%%%%%%%%%%%%%%%
%%%%%%%%%%%%%%%%%%%%%%%%%%%%%%%%%%%%

\section{An explicit semigroup defining the Gr\"obner degeneration of a toric ideal} 

Let us start by recalling the construction of a Gr\"obner degeneration of an ideal with respect to a weight vector.

Let $J\subset\pol$ be an ideal and $w=(w_1,\ldots,w_n)\in\N^n$. For $f=\sum c_ux^u\in\pol$, let $d(f):=\max\{w\cdot u|c_u\neq0\}$. Denote
$$f_t:=t^{d(f)}f(t^{-w_1}x_1,\ldots,t^{-w_n}x_n)\in\pol[t].$$
The ideal $J(t,w):=\langle f_t|f\in J\rangle$ is called the \textit{Gr\"obner degeneration of} $J$. It is a classical fact that $J(t,w)$ gives rise to a flat family deforming the affine algebraic variety defined by $J$ to the variety of the initial ideal of $J$ with respect to $w$ \cite[Theorem 15.17]{Eis}.

The following remark, which will be constantly used in this paper, provides a method to compute generators for the ideal $J(t,w)$.

\begin{rem}\label{gente}
Let $G=\{g_1,\ldots,g_s\}\subset J$ be a Gr\"obner basis of $J$ with respect to any monomial order refining $w$. Then $G(t,w)=\{(g_1)_t,\ldots,(g_s)_t\}\subset J(t,w)$ is a generating set of $J(t,w)$ \cite[Exercise 15.25]{Eis}. %Notice that if $G$ is a minimal generating set of $I$, then $G_t$ is also minimal for $I_t$. However, the converse is false in general.
\end{rem}

Our first goal is to give an explicit generating set of the semigroup defining the Gr\"obner degeneration of a toric ideal. This generating set is described in terms of generators of the original semigroup and the weight vector.

The following notation will be constantly used throughout this paper.

Let $\A=\{a_1,\ldots,a_n\}\subset\Z^d$ be a finite subset. Consider the following map of semigroups,
\begin{align}
\pi_{\A}:\N^n&\rightarrow\Z^d\notag\\
(u_1,\ldots,u_n)&\mapsto u_1a_1+\cdots+u_na_n.\notag
\end{align}
The image of $\pia$ is denoted as $\N\A$. The previous map induces a map of $\K$-algebras,
\begin{align}
\hat{\pi}_{\A}:\pol&\rightarrow\K[z_1,z_1^{-1},\ldots,z_d,z_d^{-1}]\notag\\
x_i & \mapsto z^{a_i}.\notag
\end{align}
Denote $\ia:=\ker\hat{\pi}_{\A}$. It is well known that this ideal is prime and generated by binomials. More precisely \cite[Chapter 4]{St}, 
$$\ia=\langle x^{u}-x^{v}|\pi_{\A}(u)=\pi_{\A}(v)\rangle.$$
The ideal $\ia$ is called the \textit{toric ideal defined by} $\A$.

The following theorem describes an explicit subset of $\Z^{d+1}$ that determines the Gr\"obner degeneration of $\ia$. 

\begin{teo}\label{main}
Let $\A=\{a_1,\ldots,a_n\}\subset\Z^d$ and $\ia$ the corresponding toric ideal. Let $w=(w_1,\ldots,w_n)\in\N^n$. Let 
$$\A_w:=\{(a_1,w_1),\ldots,(a_n,w_n),(0,\ldots,0,1)\}\subset\Z^{d+1}.$$
Then $\ia(t,w)=\iaw$.
\end{teo}
\begin{proof}
Let $u=(u_1,\ldots,u_n)\in\N^n$ and $u_{n+1}\in\N$. We have the following relation:
\begin{align}\label{eq}
\piw(u,u_{n+1})&=u_1(a_1,w_1)+\cdots+u_n(a_n,w_n)+u_{n+1}(0,\ldots,0,1)\notag\\
&=(u_1a_1+\cdots+u_na_n,(u,u_{n+1})\cdot (w,1))\notag\\
&=(\pi_{\A}(u),(u,u_{n+1})\cdot (w,1)).
\end{align}
Let us first show that $\ia(t,w)\subset\iaw$. Let $G\subset\ia$ be a Gr\"obner basis consisting of binomials with respect to a refined monomial order $>_w$ and such that $G(t,w)=\{g_t|g\in G\}$ generates $\ia(t,w)$ (see Remark \ref{gente}). We show that $G(t,w)\subset\iaw$.
Let $g=x^u-x^v\in G$, where $\pi(u)=\pi(v)$ and $x^u >_w x^v$. Then $g_t=x^u-x^vt^{w\cdot u-w\cdot v}$. By (\ref{eq}) it follows that
\begin{align}
\piw(v,w\cdot u-w\cdot v)&=(\pia(v),(v,w\cdot u-w\cdot v)\cdot(w,1))\notag\\
&=(\pia(v),w\cdot u)\notag\\
&=(\pia(u),w\cdot u)\notag\\
&=\piw(u,0).\notag
\end{align}
Thus, $g_t\in\iaw$ and so $\ia(t,w)\subset\iaw$.

Now we show $\iaw\subset\ia(t,w)$. Let $x^ut^{u_{n+1}}-x^vt^{v_{n+1}}\in\iaw$, i.e., $\piw(u,u_{n+1})=\piw(v,v_{n+1})$. By (\ref{eq}),
$$(\pia(u),(u,u_{n+1})\cdot(w,1))=(\pia(v),(v,v_{n+1})\cdot(w,1)).$$
In particular, $\pia(u)=\pia(v)$ and so $x^u-x^v\in\ia$. Now assume that $w\cdot u\geq w\cdot v$. Hence, $v_{n+1}=u_{n+1}+u\cdot w-v\cdot w$, and we conclude that
\begin{align}
x^ut^{u_{n+1}}-x^vt^{v_{n+1}}&=x^ut^{u_{n+1}}-x^vt^{u_{n+1}+u\cdot w-v\cdot w}\notag\\
&=t^{u_{n+1}}(x^u-x^vt^{u\cdot w-v\cdot w})\notag\\
&=t^{u_{n+1}}(x^u-x^v)_t\in\ia(t,w).\notag
\end{align}
\end{proof}

The previous theorem is our main tool to explore several combinatorial and algebraic properties of the Gr\"obner degeneration of toric ideals.

\begin{rem}
Garc\'ia-Puente, Sottile and Zhu introduced a set similar to $\A_w$ to define regular subdivisions of $\A$ \cite{GSZ,Z}.
\end{rem}

We conclude this section by defining the class of semigroups considered in this paper.

\begin{defi}\label{aff sem}
An affine semigroup $S$ is a finitely generated submonoid of $\Z^d$ such that the group generated by $S$ is $\Z^d$.
\end{defi}

\begin{rem}\label{Aw aff sem}
Given an affine semigroup generated by $\A=\{a_1,\ldots,a_n\}\subset\Z^d$ and $w\in\N^n$, we have that $\N\A_w\subset\Z^{d+1}$ is a finitely generated semigroup. In addition, since $\N\A$ generates $\Z^d$ as a group and $(0,\ldots,0,1)\in\A_w$, we obtain that $\N\A_w$ generates $\Z^{d+1}$ as a group. Hence, $\N\A_w$ is an affine semigroup.
\end{rem}

%%%%%%%%%%%%%%%%%%%%%%%%%%%%%%%%%%%%
%%%%%%%%%%%%%%%%%%%%%%%%%%%%%%%%%%%%

\section{Approximations of semigroups}

Our first application of Theorem \ref{main} concerns approximations of affine semigroups. 

A well-known result in the theory of affine semigroups states that such semigroups can be approximated by their saturation. Recall that the saturation of an affine semigroup $S\subset\Z^d$ is defined as $\overline{S}:=\R_{\geq0}S\cap\Z^d$. Equivalently, $\overline{S}=\{a\in\Z^d|ka\in S, \mbox{for some }k\in\N\}$.

\begin{teo}\cite[Theorem 1.4]{KK}
Let $S\subset\Z^d$ be an affine semigroup. There exists $a\in S$ such that $a+\overline{S}\subset S$. %, where $\sat=\R_{\geq0}\A\cap\Z^d$.
\end{teo}

Our next goal is to study this theorem for the affine semigroup of the Gr\"obner degeneration of $\ia$.

\begin{lem}\label{gen}
Let $\A=\{a_1,\ldots,a_n\}\subset\Z^d$. Assume that there exists an element $a\in\N\A$ such that $a+\sat\subset\N\A$. Let $w\in\N^n$. Consider the following sets:
\begin{align}
C&:=\Big\{\sum_{i=1}^n\alpha_ia_i|\alpha_i\in\R, 0\leq\alpha_i\leq1\Big\}.\notag\\
C_w&:=\Big\{\sum_{i=1}^n\alpha_i(a_i,w_i)+\alpha_{n+1}(0,\ldots,0,1)|\alpha_i\in\R, 0\leq\alpha_i\leq1\Big\}.\notag
\end{align}
For each element $(c,c_{n+1})\in C_w\cap\Z^{d+1}$, where $c\in\Z^{d}$, there exists $\delta\in\N$ such that $(a,\delta)\in\N\A_w$ and $(a,\delta)+(c,c_{n+1})\in\N\A_w$.
\end{lem}
\begin{proof}
Let $(c,c_{n+1})\in C_w\cap\Z^{d+1}$. Since $a\in\N\A$, we can write $a=\sum_{i=1}^nl_ia_i$. Now consider the following cases.
\begin{itemize}
\item[(1)] Suppose that $(c,c_{n+1})\in\N\A_w$. Let $\delta:=\sum_{i=1}^nl_iw_i$. Then $(a,\delta)=\sum_{i=1}^nl_i(a_i,w_i)\in\N\A_w$. In particular, $(a,\delta)+(c,c_{n+1})\in\N\A_w$.
\item[(2)] Suppose that $(c,c_{n+1})\notin\N\A_w$. Notice that $(c,c_{n+1})\in C_w\cap\Z^{d+1}$ implies that $c\in C\cap\Z^d\subset\overline{\N\A}$. In particular, $a+c\in\N\A$. Let $a+c=\sum_{i=1}^n\beta_ia_i$ for some $\beta_i\in\N$. Now consider the following cases.
\end{itemize}
\begin{itemize}
\item[(2.1)] Suppose that $(a,0)+(c,c_{n+1})\in\N\A_w$. Let $\delta:=\sum_{i=1}^nl_iw_i$. Thus $(a,\delta)=\sum_{i=1}^nl_i(a_i,w_i)\in\N\A_w$. We obtain
$$(a,\delta)+(c,c_{n+1})=(a,0)+(c,c_{n+1})+\delta(0,\ldots,0,1)\in\N\A_w.$$
\item[(2.2)] Suppose that $(a,0)+(c,c_{n+1})\notin\N\A_w$. In particular, there is no $\beta_{n+1}\in\N$ such that 
$$(a,0)+(c,c_{n+1})=\sum_{i=1}^n\beta_i(a_i,w_i)+\beta_{n+1}(0,\ldots,0,1).$$
So there is no $\beta_{n+1}\in\N$ such that $c_{n+1}=\sum_{i=1}^n\beta_iw_i+\beta_{n+1}$. Thus, $c_{n+1}<\sum_{i=1}^n\beta_iw_i$. Let $\delta:=\max\{\sum_{i=1}^n\beta_iw_i-c_{n+1},\sum_{i=1}^nl_iw_i\}$. Now notice that 
\begin{align}
(a,\delta)&=(a,\sum_{i=1}^nl_iw_i)+(\delta-\sum_{i=1}^nl_iw_i)(0,\ldots,0,1)\notag\\
&=\sum_{i=1}^nl_i(a_i,w_i)+(\delta-\sum_{i=1}^nl_iw_i)(0,\ldots,0,1)\in\N\A_w.\notag
\end{align}
Similarly,
\begin{align}
(a,\delta)+(c,c_{n+1})&=(a+c,\delta+c_{n+1})\notag\\
&=(a+c,\sum_{i=1}^n\beta_iw_i)\notag\\
&+\Big(\delta-(\sum_{i=1}^n\beta_iw_i-c_{n+1})\Big)(0,\ldots,0,1)\notag\\
&=\sum_{i=1}^n\beta_i(a_i,w_i)\notag\\
&+\Big(\delta-(\sum_{i=1}^n\beta_iw_i-c_{n+1})\Big)(0,\ldots,0,1)\in\N\A_w.\notag
\end{align}
\end{itemize}
\end{proof}

\begin{teo}\label{approx}
Let $\A=\{a_1,\ldots,a_n\}\subset\Z^d$. Assume that there exists an element $a\in\N\A$ such that $a+\sat\subset\N\A$. Then there exists $\delta\in\N$ such that $(a,\delta)\in\N\A_w$ and $(a,\delta)+\stw\subset\N\A_w$.
\end{teo}
\begin{proof}
Consider the notation of Lemma \ref{gen}. Let $\{c_1,\ldots,c_m\}=C_w\cap\Z^{d+1}$. By the lemma, there exists $\delta_i\in\N$ such that $(a,\delta_i)\in\N\A_w$ and $(a,\delta_i)+c_i\in\N\A_w$. Let $\delta:=\max_{i}\{\delta_i\}$. Let $v\in\stw$. Since $\stw=\R_{\geq0}\A_w\cap\Z^{d+1}$, we can write $v=\sum_{i=1}^nr_i(a_i,w_i)+r_{n+1}(0,\ldots,0,1)$, where $r_i\in\R_{\geq0}$. Thus,
\begin{align}
(a,\delta)+v&=(a,\delta)+\sum_{i=1}^nr_i(a_i,w_i)+r_{n+1}(0,\ldots,0,1)\notag\\
&=(a,\delta)+\sum_{i=1}^n\lfloor r_i \rfloor (a_i,w_i)+\lfloor r_{n+1} \rfloor (0,\ldots,0,1)\notag\\
&+\sum_{i=1}^n(r_i-\lfloor r_i \rfloor)(a_i,w_i)+(r_{n+1}-\lfloor r_{n+1} \rfloor)(0,\ldots,0,1).\notag
\end{align}
Notice that $\sum_{i=1}^n\lfloor r_i \rfloor (a_i,w_i)+\lfloor r_{n+1} \rfloor (0,\ldots,0,1)\in\N\A_w$ and $\sum_{i=1}^n(r_i-\lfloor r_i \rfloor)(a_i,w_i)+(r_{n+1}-\lfloor r_{n+1} \rfloor)(0,\ldots,0,1)\in C_w\cap\Z^{d+1}$. Assume that this last element is $c_1$. Then 
$$(a,\delta)+c_1=(\delta-\delta_1)(0,\ldots,0,1)+(a,\delta_1)+c_1\in\N\A_w.$$
We conclude that
$$(a,\delta)+v=\sum_{i=1}^n\lfloor r_i \rfloor (a_i,w_i)+\lfloor r_{n+1} \rfloor (0,\ldots,0,1)+(a,\delta)+c_1\in\N\A_w.$$
\end{proof}

%%%%%%%%%%%%%%%%%%%%%%%%%%%%%%%%%%%%
%%%%%%%%%%%%%%%%%%%%%%%%%%%%%%%%%%%%

\section{Saturation}\label{sat}

In this section we show that saturation of affine semigroups is not preserved by Gr\"obner degeneration. Recall that an affine semigroup $S$ is saturated if $S=\overline{S}$.

Let $m\in\N$. Consider the set
$$\A(m):=\{(1,0),(1,1),(m,m+1)\}\subset\Z^2.$$ 
It is well known that $\N\A(m)$ is a saturated semigroup. Let $w=(1,1,1)$. We show that $\N\A(m)_w$ is saturated if and only if $m\leq2$. We use this notation throughout this section.

\begin{pro}
The semigroup $\N\A(1)_w$ is saturated.
\end{pro}
\begin{proof}
Let $(a,b,c)\in\N^3$ be such that $\lambda(a,b,c)\in\N\A(1)_w$, for some $\lambda\geq1$. We want to show that $(a,b,c)\in\N\A(1)_w$. There exist $\alpha_i\in\N$ such that
\begin{align}
\lambda a&=\alpha_1+\alpha_2+\alpha_3,\notag\\
\lambda b&=\alpha_2+2\alpha_3,\notag\\
\lambda c&=\alpha_1+\alpha_2+\alpha_3+\alpha_4=\lambda a+\alpha_4.\notag
\end{align}
By the third equation, it follows that $c\geq a$. Let $c=a+d$, for some $d\geq0$. Then $(a,b,c)=(a,b,a)+d(0,0,1)$. Thus, we need to show that $(a,b,a)\in\N\A(1)_w$. Since $\lambda(a,b,c)\in\N\A(1)_w$, it follows that $\lambda(a,b)\in\N\A(1)$. Since this semigroup is saturated, $(a,b)\in\N\A(1)$, i.e., 
$$(a,b)=\beta_1(1,0)+\beta_2(1,1)+\beta_3(1,2).$$
We conclude that $(a,b,a)=\beta_1(1,0,1)+\beta_2(1,1,1)+\beta_3(1,2,1)\in\N\A(1)_w$.
\end{proof}

\begin{pro}
The semigroup $\N\A(2)_w$ is saturated.
\end{pro}
\begin{proof}
Let $b_1=(1,0,1)$, $b_2=(1,1,1)$, $b_3=(2,3,1)$, and $b_4=(0,0,1)$. By definition, $\A(2)_w=\{b_1,b_2,b_3,b_4\}$. To show that $\N\A(2)_w$ is saturated we prove that $\N\A(2)_w=\R_{\geq0}(\A(2)_w)\cap\Z^3$. Clearly, $\N\A(2)_w\subset\R_{\geq0}(\A(2)_w)\cap\Z^3$. For the other inclusion it is enough to show that $A_1\cap\Z^3\subset\N\A(2)_w$, where $A_1=\{\sum_i\lambda_ib_i|\lambda_i\in\R,0\leq\lambda_i\leq1\}$.

Let $(x,y,z)\in A_1\cap\Z^3$. Then there exist $\lambda_i\in[0,1]$ such that
\begin{align}
x&=\lambda_1+\lambda_2+2\lambda_3,\notag\\
y&=\lambda_2+3\lambda_3,\label{d}\\
z&=\lambda_1+\lambda_2+\lambda_3+\lambda_4.\notag
\end{align}
By the third equation it follows that $z\leq4$. 
\begin{itemize}
\item $z=4\Rightarrow\lambda_1=\lambda_2=\lambda_3=\lambda_4=1\Rightarrow (x,y,z)=(4,4,4)\in\N\A(2)_w$.
\item $z=0\Rightarrow \lambda_1=\lambda_2=\lambda_3=\lambda_4=0\Rightarrow (x,y,z)=(0,0,0)\in\N\A(2)_w$.
\end{itemize}
By doing elementary operations on the equations in (\ref{d}) we obtain:
\begin{align}
\lambda_3&=x-z+\lambda_4,\label{a}\\
\lambda_2+3\lambda_4&=-3x+y+3z,\label{b}\\
\lambda_1+\lambda_2+2\lambda_4&=2z-x.\label{c}
\end{align}
On the other hand, since $(x,y,z)\in A_1\cap\Z^3$ it follows that 
\begin{align}
(x,y)\in&\{\lambda_1(1,0)+\lambda_2(1,1)+\lambda_3(2,3)|\lambda_i\in[0,1]\}\cap\Z^2\notag\\
&=\{(0,0),(1,0),(1,1),(2,1),(2,2),(2,3),(3,3),(3,4),(4,4)\}.\notag
\end{align}
Now we are ready to study the remaining values of $z$.
\begin{itemize}
\item $z=1$. Of the nine options for $(x,y,z)$ only $(0,0,1)$, $(1,0,1)$, $(1,1,1)$, $(2,3,1)$ are in $A_1$ (for the other cases $-3x+y+3z<0$, contradicting (\ref{b})). These four elements belong to $\N\A(2)_w$. 
\item $z=2$. If $(0,0,2)\in A_1$ then, by equation (\ref{a}), $\lambda_4=\lambda_3+2\geq2$, which is a contradiction. Thus $(0,0,2)\notin A_1$. Similarly, $(4,4,2)\notin A_1$. On the other hand, $(2,2,2)=\frac{1}{3}(1,0,1)+(1,1,1)+\frac{1}{3}(2,3,1)+\frac{1}{3}(0,0,1)\in A_1$ and $(2,2,2)=2(1,1,1)\in\N\A(2)_w$. The other six options for $(x,y,z)$ are the sum of two different elements of $\A(2)_w$. Therefore they belong to $A_1$ and also to $\N\A(2)_w$.
\item $z=3$. If $(0,0,3)\in A_1$ then, by equation (\ref{a}), $\lambda_4=\lambda_3+3\geq3$, which is a contradiction. Thus, $(0,0,3)\notin A_1$. Similarly, $(1,0,3),(1,1,3)\notin A_1$. If $(2,2,3)\in A_1$, again by (\ref{a}),  $0\leq\lambda_4=\lambda_3+1\leq1$ implying $\lambda_3=0$ and $\lambda_4=1$. Thus, by (\ref{b}), $\lambda_2=5-3=2$, which is a contradiction. Thus, $(2,2,3)\notin A_1$. Similarly, $(2,3,3)\notin A_1$. The remaining four options for $(x,y,z)$ are the sum of three different elements of $\A(2)_w$. Therefore they belong to $A_1$ and also to $\N\A(2)_w$.
\end{itemize}
We conclude that $A_1\cap\Z^3\subset\N\A(2)_w$ and so $\R_{\geq0}(\A(2)_w)\cap\Z^3\subset\N\A(2)_w$, implying that $\N\A(2)_w$ is saturated.
\end{proof}

\begin{pro}
The semigroup $\N\A(m)_w$ is not saturated for all $m\geq3$.
\end{pro}
\begin{proof}
By definition, $\A(m)_w=\{(1,0,1),(1,1,1),(m,m+1,1),(0,0,1)\}$. Assume first that $m=2r+1$ for some $r\geq1$. Let $q=r-1\geq0$ and notice:
$$(m+1,m+1,r+1)=(1,0,1)+(m,m+1,1)+q(0,0,1)\in\N\A(m)_w.$$
On the other hand, $(2,2,1)\notin\N\A(m)_w$. Since $(m+1,m+1,r+1)=(r+1)(2,2,1)$ we conclude that $\N\A(m)_w$ is not saturated in this case.

Now assume that $m=2r$ for some $r\geq2$. Let $q=r-2\geq0$. Then,
$$(m+2,m+2,r+1)=(1,0,1)+(1,1,1)+(m,m+1,1)+q(0,0,1)\in\N\A(m)_w.$$
As before, $(m+2,m+2,r+1)=(r+1)(2,2,1)$ implying that $\N\A(m)_w$ is not saturated in this case.
\end{proof}

%%%%%%%%%%%%%%%%%%%%%%%%%%%%%%%%%%%%
%%%%%%%%%%%%%%%%%%%%%%%%%%%%%%%%%%%%

\section{Betti elements}\label{sect Betti}

As in previous sections, let $S\subset\Z^d$ be an affine semigroup generated by $\A=\{a_1,\ldots,a_n\}$. Given $w\in \N^n$, denote $S_w:=\N\A_w$. In this section we study the behaviour of Betti elements of affine semigroups under Gr\"obner degenerations. We assume that $S$ is pointed, that is, $S\cap(-S)=\{0\}$. We first recall the basic definitions we need.

Let $\sim_{\A}$ denote the kernel congruence of $\pi_{\A}$, i.e., $\alpha\sim_{\A}\beta$ if $\pi_{\A}(\alpha)=\pi_{\A}(\beta)$. It is well-known that $\sim_{\A}$ is finitely generated. A \textit{presentation} $\rho\subset \N^n\times\N^n$ for $S$ is a system of generators of $\sA$. A \textit{minimal presentation} for $S$ is a minimal system of generators of $\sA$. Notice that this is equivalent to ask for a minimal set of binomial generators of the corresponding toric ideal. 

\begin{defi}
Let $\Bet(S):=\{\pi_{\A}(\alpha)|(\alpha,\beta)\in\rho\}$, where $\rho\subset \N^n\times\N^n$ is any minimal presentation of $S$. The set $\Bet(S)$ does not depend on $\rho$ \cite[Chapter 9]{GR1}. It is called the set of Betti elements of $S$. 
\end{defi}

The terminology in the previous definition comes from the Betti numbers of $\ia$. Recall that $\ia$ is $S$-graded, where ${\rm deg}_{S}(x_i)=\pi_{\A}(e_i)$, and $e_i$ is the $i$-th element of the canonical basis of $\N^n$. The first Betti number of degree $a\in S$ of $\pls:=\plx/\ia$, denoted by $\beta_{1,a}(\pls)$ is the number of minimal generators of degree $a$ of $\ia$. It is well known that it does not depend on the minimal set of generators of $\ia$. The first Betti number of $\pls$ is the cardinality of a minimal set of generators of $\ia$, denoted by  $\beta_{1}(\pls)$, so $\beta_{1}(\pls)=\sum_{a\in S}\beta_{1,a}(\pls)$.  

\begin{rem}\label{betti-el-betti-num}
In view of the previous paragraph, for $a\in \N^{d}$, $a\in \Bet(S)$ if and only if $\beta_{1,a}(\pls)\neq 0$.
\end{rem}

Our first result relates the Betti elements of $S$ with those of $S_w$. First we prove a simple lemma.

\begin{lem}\label{min w}
Let $w\in\N^ n$. Let $\{g_1,\ldots,g_s\}\subset\ia$ be a generating set of $\ia$. Assume that $\{(g_1)_t,\ldots,(g_r)_t\}$ generates $\iaw$, for some $r\leq s$. Then $\{g_1,\ldots,g_r\}$ generates $\ia$.
\end{lem}
\begin{proof}
If $r=s$, there is nothing to prove. Suppose $r<s$ and let $j\in\{r+1,\ldots,s\}$. By hypothesis, $(g_j)_t=\sum_{i=1}^r h_{i}(x,t)(g_i)_t.$ Making $t=1$ it follows $g_j=\sum_{i=1}^r h_{i}(x,1)g_i$.
\end{proof}

\begin{teo}\label{inclusion}
Let $w\in\N^n$. Then, for each $b\in\Bet(S)$ there exists $\lambda\in\N$ such that $(b,\lambda)\in\Bet(S_w)$.
%Let $p:\N^{d+1}\to\N^d$ is the projection into the first $d$ coordinates. 
\end{teo}
\begin{proof}
Let $\{g_1,\ldots,g_s\}\subset \ia$ be a Gr\"obner basis with respect to a refined order $>_w$, where $g_i=x^{\alpha_i}-x^{\beta_i}$ and $\alpha_i\cdot w\geq\beta_i\cdot w$, for each $i$. Then, $\iaw=\langle (g_1)_t,\ldots,(g_s)_t \rangle$ by Remark \ref{gente}. 

After reordering the $g_i$'s if necessary, we may assume that $\{(g_1)_t,\ldots,(g_r)_t\}$ is a minimal generating set, for some $r\leq s$. In particular, $\Bet(S_w)=\{\pi_{\A_w}(\alpha_i,0)=(\pi_{\A}(\alpha_i),\alpha_i\cdot w)\}_{i=1}^r$. On the other hand, by Lemma \ref{min w}, it follows that $\{g_1,\ldots,g_r\}$ generates $\ia$ and so this set contains a minimal generating set. This implies the theorem.
\end{proof}

\begin{rem}
A similar result to Theorem \ref{inclusion} was proved in \cite[Theorem 8.29]{MS}.
\end{rem}

In view of Theorem \ref{inclusion}, several natural questions arise. Is the element $\lambda$ unique? Does every Betti element of $S_w$ have as first coordinate a Betti element of $S$? In the following examples we show that every scenario could actually happen. Example \ref{1-1} shows that $\lambda$ may not be unique. Example \ref{scroll} shows that there are Betti elements of $S_w$ whose first coordinate is not a Betti element of $S$. Finally, Example \ref{un betti w} exhibits an infinite family of numerical semigroups where the map $\Bet(S)\to\Bet(S_w)$, $b\mapsto(b,\lambda)$ is well-defined and  bijective. These examples show that Theorem \ref{inclusion} is the best result we can expect regarding Betti elements of $S$ and $S_w$.

\begin{exam}\label{1-1}
Let $a_1=6, a_2=10, a_3=15$ and $\A=\{a_1,a_2,a_3\}$. Then $\ia$ is minimally generated by $\{x_1^5-x_2^3,x_2^3-x_3^2\}$. Thus, $\Bet(S)=\{30\}$. Let $w=(1,1,1)$. The previous set of generators is also a Gr\"obner basis for $\ia$ with respect to $>_w$, where $>$ is the lexicographical order. Hence, $\iaw$ is generated by $\{x_1^5-x_2^3t^2,x_2^3-x_3^2t\}$. Actually, it is a minimal generating set. In particular, $\Bet(S_w)=\{(30,5),(30,3)\}$.
\end{exam}

\begin{exam}\label{scroll}
Let $\A=\{(1,0),(1,1),(1,2),(1,3)\}\subset\N^2$. Then $\ia$ is minimally generated by $\{ac-b^2,ad-bc,bd-c^2\}$. Thus, $\Bet(S)$ is the set $\{(2,2),(2,3),(2,4)\}$. Let $w=(3,7,2,5)$. A Gr\"obner basis for $\ia$ with respect to $>_{w}$, where $>$ is the lexicographical order, is $\{b^2-ac,bc-ad,bd-c^2,ad^2-c^3\}$. Hence, $\iaw$ is generated by $\{b^2-act^9,bc-adt,bd-c^2t^8,ad^2-c^3t^7\}.$ Actually, it is a minimal generating set. In particular, $\Bet(S_w)=\{(2,2,14),(2,3,9),(2,4,12),(3,6,13)\}$.
\end{exam}

\begin{exam}\label{un betti w}
%Let $S=\langle a_1,\ldots,a_n\rangle\subset\N$ be such that $|\Bet(S)|=1$. There exist
Let $b_1,\ldots,b_n\geq2$ be pairwise relatively prime integers and $a_i:=\prod_{j\neq i}b_j$. Let $\A=\{a_1,\ldots,a_n\}$ and $S=\N\A\subset\N$. In this case, $\Bet(S)=\{b\}$, where $b=\prod_{j=1}^n b_j$ \cite[Example 12]{GOR}. Let $w=(w_1,\ldots,w_n)\in\N^n$. Consider a permutation $(i_1,\ldots,i_n)$ of $(1,\ldots,n)$ such that $b_{i_1}w_{i_1}\geq b_{i_2}w_{i_2}\geq\cdots\geq b_{i_n}w_{i_n}$.

It is known that $\ia=\langle f_2,\ldots,f_n\rangle$, where $f_i=x_1^{b_1}-x_i^{b_i}$, for each $i\in\{2,\ldots,n\}$ \cite[Example 12]{GOR}. Define
\begin{align}
&g_{i_2}=x_{i_1}^{b_{i_1}}-x_{i_2}^{b_{i_2}}, &&h_{i_2}=x_{i_1}^{b_{i_1}}-x_{i_2}^{b_{i_2}},\notag\\
&g_{i_3}=x_{i_2}^{b_{i_2}}-x_{i_3}^{b_{i_3}},  &&h_{i_3}=x_{i_1}^{b_{i_1}}-x_{i_3}^{b_{i_3}},\notag\\
&\vdots &&\vdots \notag\\
&g_{i_n}=x_{i_{n-1}}^{b_{i_{n-1}}}-x_{i_n}^{b_{i_n}},  &&h_{i_n}=x_{i_1}^{b_{i_1}}-x_{i_n}^{b_{i_n}}.\notag
\end{align}
A direct computation shows that 
$$\langle g_{i_2},\ldots,g_{i_n} \rangle=\langle h_{i_2},\ldots,h_{i_n} \rangle=\langle f_2,\ldots,f_n \rangle.$$
Notice that the binomials $f_i$, $h_{i_j}$, and $g_{i_j}$ are $S$-homogeneous. In addition, $\{f_2,\ldots,f_n\}$ is a minimal generating set of $\ia$. Indeed, assume, for instance, that $f_2$ can be written in terms of $f_3,\ldots,f_n$. Then, evaluating at $(0,1,0,\ldots,0)$ we obtain a contradiction. By the cardinality of $\{g_{i_2},\ldots,g_{i_n}\}$, it follows that this set is also a minimal generating set of $\ia$.

Let $>$ denote the lexicographical order on $\K[x_1,\ldots,x_n]$ with the variables ordered as $x_{i_1}>\cdots>x_{i_n}$. Then, for each $j\in\{2,\ldots,n\}$, we have $\LT_{>_{w}}(g_{i_j})=x_{i_{j-1}}^{b_{i_{j-1}}}$. In particular, for any $1<j<k\leq n$, the leading terms $\LT_{>_{w}}(g_{i_j})$ and $\LT_{>_{w}}(g_{i_k})$ are relatively prime. Hence, $\{g_{i_2},\ldots,g_{i_n}\}$ is a Gr\"obner basis of $\ia$ with respect to $>_{w}$ \cite[Corollary 2.3.4]{HH}. It follows that
$$\iaw=\langle (g_{i_2})_t,\ldots,(g_{i_n})_t \rangle.$$
Notice that each $(g_{i_j})_t=x_{i_{j-1}}^{b_{i_{j-1}}}-x_{i_{j}}^{b_{i_{j}}}t^{w_{i_{j-1}}b_{i_{j-1}}-w_{i_{j}}b_{i_{j}}}$ is $S_w$-homogeneous of degree $(b,b_{i_{j-1}}w_{i_{j-1}})$.

We already proved that $\{g_{i_2},\ldots,g_{i_n}\}$ is a minimal generating set of $\ia$. Therefore, $\{(g_{i_2})_t,\ldots,(g_{i_n})_t\}$ is also a minimal generating set of $\iaw$ (see Remark \ref{gente} and Lemma \ref{min w}). We conclude that $\Bet(S_w)$ equals the set $\{(b,b_{i_{1}}w_{i_{1}}),\ldots,(b,b_{i_{n-1}}w_{i_{n-1}})\}.$ In particular, $|\Bet(S_w)|=1$ if and only if $b_{i_1}w_{i_1}=\cdots=b_{i_{n-1}}w_{i_{n-1}}$. In this case, the map $\Bet(S)\to\Bet(S_w)$, $b\mapsto(b,b_{i_1}w_{i_1})$ is well-defined and bijective.
\end{exam}

%%%%%%%%%%%%%%%%%%%%%%%%%%%%%%%%%%%%

\subsection{Semigroups with unique minimal generating set}\label{unique min}

In this section we explore further consequences of Theorem \ref{inclusion} in the context of affine semigroups having a unique minimal generating set. Such semigroups have been studied, for instance, in \cite{CKT,GO,OV}. 

We say that an affine semigroup $S$ is \textit{uniquely presented} if it has a unique minimal presentation. Notice that this is equivalent to the corresponding toric ideal having a unique minimal generating set of binomials, up to scalar multiplication.

\begin{rem}
The notion of \textit{uniquely presented} is not to be confused with the previous notion of \textit{minimally presented}. The former asks for a unique minimal generating set of a toric ideal whereas the latter just asks for a minimal generating set, up to scalar multiplication. 
\end{rem}

We define a partial order $>_S$ on $S$ as follows (recall that we assume $S\cap(-S)=\{0\}$): $\alpha>_S\beta$ if $\alpha-\beta\in S$. We say that $\alpha\in\Bet(S)$ is \textit{Betti minimal} if it is minimal with respect to the order $>_S$. The set of such elements is denoted as $\Betm(S)$.

\begin{rem}\label{betti vs unique}
It is known that $\ia$ is uniquely presented if and only if $\Bet(S)=\Betm(S)$ and the cardinality of $\Bet(S)$ is equal to the cardinality of a minimal generating set of binomials of $\ia$ \cite[Corollary 6]{GO}. By Remark \ref{betti-el-betti-num}, this is also equivalent to $\Bet(S)=\Betm(S)$ and $\beta_{1,a}(\plx/\ia)=1$ for all $a\in \Bet(S)$.
\end{rem}

In the following proposition we show that, for uniquely presented semigroups, the Betti elements of $S_w$ coming from Betti elements of $S$ are Betti-minimal.

Let $S$ be an affine semigroup with Betti elements $\Bet(S)=\{b_1,\ldots,b_r\}$. Let $w\in\N^n$. By Theorem \ref{inclusion}, there exist some $\lambda_i\in\N$ such that $(b_i,\lambda_i)\in\Bet(S_w)$, for each $i\in\{1,\ldots,r\}$.

\begin{teo}\label{betti min}
With the previous notation, assume in addition that $S$ is uniquely presented. Then each $(b_i,\lambda_i)$ is Betti-minimal.
\end{teo}
\begin{proof}
Let $\{g_1,\ldots,g_r\}\subset\ia$ be the only minimal binomial generating set of $\ia$. Using Buchberger's algorithm, we extend this set to a Gr\"obner basis $\{g_1,\ldots,g_r,g_{r+1},\ldots,g_s\}$ with respect to $>_w$. In particular, $g_{r+1},\ldots,g_s$ are binomials as well. We can assume that for each $j\in\{r+1,\ldots,s\}$, the binomial $g_j$ is not an scalar multiple of $g_i$, for all $i\in\{1,\ldots,r\}$. 

We have that $\{(g_1)_t,\ldots,(g_s)_t\}$ generates $\iaw$. We claim that, for $i\in\{1,\ldots,r\}$, the binomial $(g_i)_t$ is not generated by $\{(g_j)_t\}_{j\neq i}$. Indeed, suppose this is the case for some $i$. Lemma \ref{min w} implies that $\{g_j\}_{j\neq i}$ is a generating set of $\ia$. By removing redundant elements of this set, we obtain a minimal binomial generating set of $\ia$ not containing $g_i$. This contradicts the uniqueness of the minimal binomial generating set of $\ia$. 

Consider any subset of $\{(g_i)_t\}_{i=1}^s$ that minimally generates $\iaw$. By the previous paragraph, such a subset must contain $\{(g_1)_t,\ldots,(g_r)_t\}$. For each $i\in\{1,\ldots,r\}$ we write $(g_i)_t=x^{\alpha_i}-x^{\beta_i}t^{\tau_i}, \mbox{ for some }\tau_i\in\N,$ where $\pi_{\A_w}(\alpha_i,0)=\pi_{\A_w}(\beta_i,\tau_i)=(b_i,\lambda_i)\in\Bet(S_w)$.

By the uniqueness condition on $S$, each $b_i\in\Bet(S)$ is Betti-minimal and $\pi_{\A}^{-1}(b_i)=\{\alpha_i,\beta_i\}$ for each $i\in\{1,\ldots,r\}$ \cite[Section 3]{GO}. This implies that 
$$\pi_{\A_w}^{-1}(b_i,\lambda_i)=\{(\alpha_i,0),(\beta_i,\tau_i)\}.$$
Indeed, if $\pi_{\A_w}(\gamma,l)=(b_i,\lambda_i)$ then $\gamma=\alpha_i$ or $\gamma=\beta_i$ implying that $l=0$ or $l=\tau_i$, respectively.

Summarizing, for each $i\in\{1,\ldots,r\}$ we have $(b_i,\lambda_i)\in\Bet(S_w)$ and $|\pi_{\A_w}^{-1}(b_i,\lambda_i)|=2$. We conclude that $(b_i,\lambda_i)$ is Betti-minimal \cite[Corollary 5]{GO}.
\end{proof}

\begin{coro}\label{all betti min}
Let $S$ be a uniquely presented affine semigroup. Let $w\in\N^n$ be such that some minimal generating set of $\ia$ is also a Gr\"obner basis with respect to some refined order $>_w$. Then all Betti elements of $S_w$ are Betti-minimal. In particular, $\iaw$ is also uniquely presented.
\end{coro}
\begin{proof}
Let $\{g_1,\ldots,g_r\}\subset \ia$ be a minimal generating set that is also a Gr\"obner basis with respect to $>_w$. Then $\{(g_1)_t,\ldots,(g_r)_t\}\subset \iaw$ is a generating set. In addition, by Lemma \ref{min w}, it is minimal. Thus, every Betti element of $S_w$ is of the form $(b,\lambda)$, for some $\lambda\in\N$ and $b\in\Bet(S)$. By Theorem \ref{betti min}, such Betti elements are Betti-minimal. The last statement of the corollary follows from Remark \ref{betti vs unique}.
\end{proof}

Let us look at an example where the conditions of the previous corollary are satisfied.

\begin{exam}\label{exam uniq pres}
Let $S=\langle a,a+1,a+2 \rangle\subset\N$, where $a=2q\geq4$. The only minimal set of binomial generators of $\ia$ is $\{y^2-xz,x^{q+1}-z^q\}$ \cite[Theorem 15]{GO}. Let $w\in\N^3$ be such that $2w_2>w_1+w_3$. Then the leading terms of these two binomials with respect to any refined order $>_w$ are relatively prime. Hence, they form a Gr\"obner basis with respect to $>_w$. By Corollary \ref{all betti min}, $\iaw$ is uniquely presented.
\end{exam}

Now we show an example of a uniquely presented affine semigroup such that \textit{any} of its Gr\"obner degenerations is also uniquely presented.

As usual, let $\A=\{a_1,\ldots,a_n\}$. The Lawrence ideal of $\A$, denoted $I_{\Lambda(\A)}$, is the ideal of $\K[x_1,\ldots,x_n,y_1,\ldots,y_n]$ generated by
$$\Big\{x^uy^v-x^vy^u|\sum_{i=1}^{n}u_ia_i=\sum_{i=1}^{n}v_ia_i\Big\}.$$ 
This ideal is studied, for instance, in \cite[Chapter 7]{St}. There, Lawrence ideals are used as an auxiliary tool to compute Graver bases. The relevant fact for us is that Lawrence ideals are uniquely presented \cite[Corollary 16, Proposition 4]{OV}.

\begin{coro}\label{Lawrence}
Any Gr\"obner degeneration of the Lawrence ideal is uniquely presented.
\end{coro}
\begin{proof}
It is known that any minimal binomial generating set $\{g_1,\ldots,g_r\}$ of $I_{\Lambda(\A)}$ is a reduced Gr\"obner basis with respect to any order \cite[Theorem 7.1]{St}. By Corollary \ref{all betti min}, any Gr\"obner degeneration of the Lawrence ideal is uniquely presented.
\end{proof}

%%%%%%%%%%%%%%%%%%%%%%%%%%%%%%%%%%%%

\subsection{A further example of Gr\"obner degenerations preserving the uniqueness of a presentation}\label{unique min 2}

In this section we present a further example showing that Gr\"obner degenerations preserve the property of being uniquely presented. We stress that, as opposed to Example \ref{exam uniq pres} or Corollary \ref{Lawrence}, the results of this section do not rely on Corollary \ref{all betti min}.

\begin{pro}\label{grob1}
Let $S=\langle a,a+1,a+2 \rangle$, $a\in\N$, $a=2q+2$, $q\geq1$, and $w=(w_1,w_2,w_3)\in\N^3\setminus\{(0,0,0)\}$. The following are Gr\"obner bases of $\ia$ for the $w$-degrevlex order (the leading monomials are listed first):
\begin{enumerate}
\item Suppose $2w_2\geq w_1+w_3$. Then $G=\{y^2-xz,x^{q+2}-z^{q+1}\}.$

\item Suppose $2w_2<w_1+w_3$ and $(q+1)w_3\leq (q+2)w_1$. 
\begin{enumerate}
		\item Suppose $(q+2)w_3\leq (q+1)w_1+2w_2$. Then $G=\{xz-y^2,x^{q+2}-z^{q+1},x^{q+1}y^2-z^{q+2}\}.$
		\item Suppose $(q+1)w_1+2w_2<(q+2)w_3$. Then $G=\{xz-y^2,x^{q+2}-z^{q+1},z^{q+2}-x^{q+1}y^2\}.$
	\end{enumerate} 
\item Suppose $2w_2<w_1+w_3$, $(q+2)w_1<(q+1)w_3$.
\begin{enumerate}
\item\label{impor} If $(q+i+3)w_1<2(i+1)w_2+(q-i)w_3$ for all $0\leq i\leq q$, then
$G=\{xz-y^2,z^{q+1}-x^{q+2},y^2z^q-x^{q+3},y^4z^{q-1}-x^{q+4},\ldots,y^{2(q+1)}-x^{2q+3}\}$.
\item\label{impor2} Suppose there exists $n\in\N$, $n\leq q$ such that $2(n+1)w_2+(q-n)w_3\leq (q+n+3)w_1$, and for all $0\leq i <n$, $(q+i+3)w_1<2(i+1)w_2+(q-i)w_3$. Then 
$G=\{xz-y^2,z^{q+1}-x^{q+2},y^2z^q-x^{q+3},y^4z^{q-1}-x^{q+4},\ldots,y^{2n}z^{q-(n-1)}-x^{q+(n-1)+3},x^{q+n+3}-y^{2(n+1)}z^{q-n}\}.$
\end{enumerate}
\end{enumerate}
\end{pro}
\begin{proof}
Let $p_1=y^2-xz$ and $p_2=x^{q+2}-z^{q+1}$. A straightforward computation shows that $\ia=\langle p_1,p_2 \rangle$. More generally, generators for toric ideals of semigroups generated by intervals can be found in \cite[Theorem 8]{GR2}.
\begin{enumerate}
	\item Suppose $2w_2\geq w_1+w_3$. Then the initial monomials of $p_1$ and $p_2$ are relatively prime. Hence, $G=\{p_1,p_2\}$ is a Gr\"obner basis of $\ia$.

	\item	Suppose $2w_2<w_1+w_3$ and $(q+1)w_3\leq (q+2)w_1$. The corresponding $S$-polynomial of $p_1$ and $p_2$ is $S(p_1,p_2)=z^{q+2}-x^{q+1}y^2$. An application of Buchberger's algorithm shows that $G=\{p_1,p_2,S(p_1,p_2)\}$ is a Gr\"obner basis of $\ia$, whether $z^{q+2}$ or $x^{q+1}y^2$ is the initial monomial of $S(p_1,p_2)$.
	\item Suppose $2w_2<w_1+w_3$, $(q+2)w_1<(q+1)w_3$. Consider $p_3:=S(p_1,p_2)=x^{q+3}-y^2z^q$.
   
   If $(q+3)w_1\geq 2w_2+qw_3$ then, by Buchberger's algorithm, $G=\{p_1,p_2,p_3\}$ is a Gr\"obner basis of $\ia$. 
   
   Suppose $(q+3)w_1< 2w_2+qw_3$. Let $p_4:=S(p_1,p_3)=x^{q+4}-y^4z^{q-1}$. In addition, notice $S(p_2,p_3)=x^{q+3}z-x^{q+2}y^{2}=x^{q+2}p_1$.
			
    If $(q+4)w_1\geq 4w_2+(q-1)w_3$ then, by Buchberger's algorithm, $G=\{p_1,p_2,p_3,p_4\}$ is a Gr\"obner basis of $\ia$.

    Suppose $(q+4)w_1< 4w_2+(q-1)w_3$. If $q=1$, then by Buchberger's algorithm $G=\{p_1,p_2,p_3,p_4\}$ is a Gr\"obner basis of $\ia$. Otherwise $q\geq 2$. Let $p_5=S(p_1,p_4)=x^{q+5}-y^6z^{q-2}$. In addition, notice $S(p_2,p_4)=z^{2}x^{q+4}-y^4x^{q+2}=(x^{q+3}z+x^{q+2}y^{2})p_1$ and $S(p_3,p_4)=zx^{q+4}-y^{2}x^{q+3}=x^{q+3}p_1$.				
				
    If $(q+5)w_1\geq 6w_2+(q-2)w_3$ then, by Buchberger's algorithm, $G=\{p_1,p_2,p_3,p_4,p_5\}$ is a Gr\"obner basis of $\ia$.

    Suppose $(q+5)w_1< 6w_2+(q-2)w_3$. If $q=2$, then by Buchberger's algorithm $G=\{p_1,p_2,p_3,p_4,p_5\}$ is a Gr\"obner basis of $\ia$. Otherwise we continue as before and we have two possibilities:      
		\begin{itemize}
			\item[i.] For all $n\in\N$, $n\leq q$, $(q+i+3)w_1<2(i+1)w_2+(q-i)w_3$. In this case we obtain the statement \ref{impor} of this proposition.
			\item[ii.] There exists $n\in \N$, $n<q$ such that $(q+n+3)w_1\geq 2(n+1)w_2+(q-n)w_3$, and for all $0\leq i <n$, $(q+i+3)w_1<2(i+1)w_2+(q-i)w_3$. In this case we obtain the statement \ref{impor2} of this proposition.
		\end{itemize}
\end{enumerate}
\end{proof}

\begin{lem}\label{mon-deg}
Let $\ia$ be a toric ideal and $G_0=\left\{ x^{\alpha_i}-x^{\beta_i}:i\in\{1,\dots,m\}\right\}$ be a minimal generating set of $\ia$. Let $M_S=\{x^{\alpha_i},x^{\beta_i}:i\in\{1,\dots, m\}\}$. Suppose that the ideal $\langle M_S\rangle$ is minimally generated by $M_S$. In addition, suppose that $\pi_{\A}(\alpha_i)\neq \pi_{\A}(\alpha_j)$, for all $i,j\in\{1,\dots,m\}$, $i\neq j$. Then $\ia$ is uniquely presented and $G_0$ is the unique minimal generating set of binomials, up to scalar multiplication.
\end{lem}
\begin{proof}
By \cite[Proposition 3.1]{CKT}, for any monomial $m\in M_S$ and for any minimal generating set of binomials $G$ of $\ia$, there exists $\gamma\in \N^n$, such that $m-x^{\gamma}\in G$ or $x^{\gamma}-m\in G$. 

Suppose there exists $l\in\{1,\dots,m\}$ such that $x^{\alpha_l}-x^{\beta_l}\notin G$ and $x^{\beta_l}-x^{\alpha_l}\notin G$, for some minimal generating set of binomials $G$ of $\ia$. As $\pi_{\A}(\alpha_l)=\pi_{A}(\beta_l)$, by Remark \ref{betti-el-betti-num} $\beta_{1,\pi_{\A}(\alpha_l)}(\plx/\ia)\geq 2$. 

By hypothesis, $\pi_{\A}(\alpha_i)\neq\pi_{\A}(\alpha_j)$ for all $i\neq j$, and $G_0$ is a minimal generating set of $\ia$. By Remark \ref{betti-el-betti-num} $\beta_{1,\pi_{\A}(\alpha_i)}(\plx/\ia)=1$, for each $i\in\{1,\ldots,m\}$. This is a contradiction. Thus $G_0$ is the unique minimal generating set of binomials, up to scalar multiplication.
\end{proof}

\begin{pro}\label{min grob1}
Let $S=\langle a,a+1,a+2 \rangle$, $a\in\N$, $2<a$ and $a=2q+2$. Then, for any $w\in\N^3\setminus\{(0,0,0)\}$, the ideal $\iaw$ is uniquely presented.
\end{pro}
\begin{proof}
We divide the proof based on the three cases of Proposition \ref{grob1}. The ideal $\ia$ is uniquely presented by \cite[Theorem 15]{GO}. In addition, its unique minimal generating set is $\{y^2-xz,x^{q+2}-z^{q+1}\}$ (see the proof of Proposition \ref{grob1}). Hence, case 1 follows from Corollary \ref{all betti min}.

Case 2(a). From the Gr\"obner basis of $\ia$ we can produce a generating set of $\iaw$. The monomials appearing in this generating set are:
\begin{align}
\{&xz,y^2t^{w_1+w_3-2w_2},x^{q+2},z^{q+1}t^{(q+2)w_1-(q+1)w_3},\notag\\
&x^{q+1}y^2,z^{q+2}t^{(q+1)w_1+2w_2-(q+2)w_3}\}.\notag
\end{align}
By the inequalities satisfied by $w$ in this case, it follows that this set is a minimal generating set of the ideal it generates. In addition, the $\A_w$-degrees of these monomials are:
$$\{(2a+2,w_1+w_3),((q+2)a,(q+2)w_1),((q+2)(a+2),(q+1)w_1+2w_2)\}.$$
Since these degrees are different, we conclude that $\iaw$ is uniquely presented by Lemma \ref{mon-deg}.

Case 2(b). As in the previous case, we obtain the following set of monomials:
\begin{align}
\{&xz,y^2t^{w_1+w_3-2w_2},x^{q+2},z^{q+1}t^{(q+2)w_1-(q+1)w_3},\notag\\
&z^{q+2},x^{q+1}y^2t^{(q+2)w_3-(q+1)w_1-2w_2}\}.\notag
\end{align}
If $(q+2)w_1>(q+1)w_3$, then proceed exactly as in case 2(a). Suppose that $(q+2)w_1=(q+1)w_3$. Then $(q+2)w_3-(q+1)w_1-2w_2=w_1+w_3-2w_2=:\delta.$
Thus, $z^{q+2}-x^{q+1}y^2t^{\delta}=x^{q+1}(xz-y^2t^{\delta})-z(x^{q+2}-z^{q+1}).$ This implies that $\iaw$ is generated by $\{xz-y^2t^{\delta},x^{q+2}-z^{q+1}\}$. Hence, $\iaw$ is uniquely presented by Lemma \ref{mon-deg}.

Case 3(a). As before, we obtain the following set of monomials:
\begin{align}
\{&xz,y^2t^{w_1+w_3-2w_2},z^{q+1},x^{q+2}t^{(q+1)w_3-(q+2)w_1}\}\notag\\
&\cup\{y^{2(i+1)}z^{q-i},x^{q+i+3}t^{2(i+1)w_2+(q-i)w_3-(q+i+3)w_1}\}_{0 \leq i \leq q}.\notag
\end{align}
By the inequalities satisfied by $w$ in this case, it follows that this set is a minimal generating set of the ideal it generates. It remains to prove that the $\A_w$-degrees are all different. It is enough to show this for the first entry of the $\A_w$-degrees. Indeed, the first entries are 
$$\{2a+2,(q+2)a,\ldots,(q+i+3)a,\ldots,(2q+3)a\}.$$
We conclude that $\iaw$ is uniquely presented by Lemma \ref{mon-deg}.

Case 3(b). Proceed exactly as in case 3(a).
\end{proof}

\begin{rem}\label{conjdAD}
By computing Gr\"obner bases and using Lemma \ref{mon-deg}, we verified that Proposition \ref{min grob1} also holds for other families of uniquely presented numerical semigroups generated by intervals. Moreover, we used the same method to study this property for other families of numerical semigroups. Our computations give enough evidence to conjecture that the uniqueness of a presentation of a toric ideal is preserved under Gr\"obner degenerations.
\end{rem}

%%%%%%%%%%%%%%%%%%%%%%%%%%%%%%%%%%%%
%%%%%%%%%%%%%%%%%%%%%%%%%%%%%%%%%%%%

\section{M\"obius functions}

In this final section we study M\"obius functions of affine semigroups. Several authors have provided explicit formulas for M\"obius functions of some families of semigroups. In particular, the case of semigroups with a unique Betti element was studied in \cite{CGMR}. As a final application of Theorem \ref{main}, we present some relations among the M\"obius functions of $S$ and $S_w$, in the case where both $S$ and $S_w$ have a unique Betti element.

Let $S\subset\Z^d$ be a pointed affine semigroup. As in previous sections, consider the following partial order: for $x,y\in\Z^d$, $x <_S y$ if $y-x\in S$. An interval on $\Z^d$ with respect to $<_S$ is defined as $[x,y]_{\Z^d}:=\{z\in\Z^d|x\leq_{S}z\leq_{S}y\}.$
Denote as $c_l(x,y)$ the cardinality of 
$$\{\{x<_{S}z_1<_{S}z_2<_{S}\cdots<_{S}z_l=y\}\subset[x,y]_{\Z^d}\}.$$
The M\"obius function of $S$, denoted $\mu_S$, is defined as 
$$\mu_S:\Z^d\to\Z, \mbox{ }\mbox{ } y\mapsto\sum_{l\geq0}(-1)^lc_l(0,y).$$
This sum is always finite \cite[Section 2]{CGMR}. Notice that if $y\notin S$ then $\mu_S(y)=0$ (since, in this case, $c_l(0,y)=0$ for all $l\geq0$). Thus, we restrict the domain of $\mu_S$ to $S$.%, $\mu_S:S\to\Z$, $y\mapsto\mu_S(y).$

The following formula is the starting point of our discussion.

\begin{teo}\cite[Theorem 4.1]{CGMR}\label{formula}
Let $S=\langle a_1,\ldots,a_n\rangle\subset\Z^{d}$ be a pointed affine semigroup. Suppose that $\Bet(S)=\{b\}$. Then,
$$\mu_S(z)=\sum_{j=1}^t(-1)^{|A_j|}\binom{k_{A_j}+n-d-1}{k_{A_j}},$$
where $\{A_1,\ldots,A_t\}=\{A\subset\{1,\ldots,n\}|\exists k_{A}\in\N \mbox{ such that }z=\sum_{i\in A}a_i+k_Ab\}$.
\end{teo}

We introduce some notation that we use throughout this section.

Let $w\in\N^n$ and $S_w=\langle a'_1,\ldots,a'_n,a'_{n+1}\rangle\subset\Z^{d+1}$, where $a'_i=(a_i,w_i)$ for $i\in\{1,\ldots,n\}$ and $a'_{n+1}=(\zz,1)$. Assume that $\Bet(S)=\{b\}$ and $\Bet(S_w)=\{(b,d_w)\}$. For $z\in S$ and $(z,\lambda)\in S_w$, denote:
\begin{itemize}
\item $\az:=\{A\subset\{1,\ldots,n\}|\exists k_{A}\in\N \mbox{ such that }z=\sum_{i\in A}a_i+k_Ab\}$.
\item $\bz:=\{B\subset\{1,\ldots,n+1\}|\exists k_{B}\in\N \mbox{ such that }(z,\lambda)=\sum_{i\in B}a'_i+k_B(b,d_w)\}$.
\end{itemize}

\begin{lem}\label{b-a}
Let $z\in S$ and $(z,\lambda)\in S_w$. Then,
\begin{itemize}
\item[(i)] $\bz':=\{B\setminus\{n+1\}|B\in\bz\}\subset\az$. In addition, if $B\setminus\{n+1\}=A$, for some $A\in\az$, then $k_{B}=k_A$.
\item[(ii)] If $B\neq C$ in $\bz$ then $B\setminus\{n+1\}\neq C\setminus\{n+1\}$. In particular, $|\bz|=|\bz'|\leq|\az|$.
\end{itemize}
\end{lem}
\begin{proof}
Let $B\in\bz$. There exists $k_B\in\N$ such that $(z,\lambda)=\sum_{i\in B}a'_i+k_B(b,d_w)$. In particular,
$$z=\sum_{i\in B}a_i+k_Bb=\sum_{i\in B\setminus\{n+1\}}a_i+k_Bb.$$
Thus, $B\setminus\{n+1\}\in\az$. This shows the first part of $(i)$. For the same $B$, let $A\in\az$ be such that $A=B\setminus\{n+1\}$. Then $\sum_{i\in B\setminus\{n+1\}}a_i+k_Bb=z=\sum_{i\in A}a_i+k_Ab$. Hence, $k_A=k_B$.

Now we prove $(ii)$. The result is clear if $n+1$ is contained in $B$ and $C$ or if it is not contained in either. Thus, we can assume that $n+1\in B$ and $n+1\notin C$. Suppose that $B\setminus\{n+1\}=C\setminus\{n+1\}$. In particular, there exist $k_B,k_C\in\N$ such that $\sum_{i\in B\setminus\{n+1\}}a_i+k_Bb=z=\sum_{i\in C\setminus\{n+1\}}a_i+k_Cb$. Hence $k_B=k_C$. This implies
$$\sum_{i\in B}a_i'=(z,\lambda)-k_B(b,d_w)=(z,\lambda)-k_C(b,d_w)=\sum_{i\in C}a_i'.$$
In particular, $\sum_{i\in B}w_i=\sum_{i\in C}w_i$. This is a contradiction since, by the assumption on $C$,
$$\sum_{i\in C}w_i=\sum_{i\in C\setminus\{n+1\}}w_i<1+\sum_{i\in C\setminus\{n+1\}}w_i=1+\sum_{i\in B\setminus\{n+1\}}w_i=\sum_{i\in B}w_i.$$
\end{proof}

\begin{lem}\label{lambdas}
Let $z\in S$. Suppose that $\az=\{A_1,\ldots,A_r\}$, $r\geq1$. Let $l_j=\sum_{i\in A_j}w_i+k_{A_j}d_w$, for each $j\in\{1,\ldots,r\}$. Let $\lambda\in\N$ be such that $(z,\lambda)\in S_w$. Then,
$$\bz\neq\emptyset\Longleftrightarrow \lambda=l_j\mbox{ or }\lambda=l_j+1 \mbox{ for some }j\in\{1,\ldots,r\}.$$
\end{lem}
\begin{proof}
Let $B\in\bz$. We have two cases:
\begin{itemize}
\item $n+1\notin B$. By $(i)$ of Lemma \ref{b-a}, $B=A_j$ for some $j$. Hence, $(z,\lambda)=\sum_{i\in B}a_i'+k_B(b,d_w)=\sum_{i\in A_j}a_i'+k_B(b,d_w)$. We also know that $k_B=k_{A_j}$. We conclude that $\lambda=\sum_{i \in B}w_i+k_Bd_w=\sum_{i\in A_j}w_i+k_{A_j}d_w=l_j$.
\item $n+1\in B$. Like in the previous item, $B\setminus\{n+1\}=A_j$ and $k_B=k_{A_j}$, for some $j$. It follows that $\lambda=\sum_{i\in B}w_i+k_Bd_w=(1+\sum_{i\in A_j}w_i)+k_{A_j}d_w=1+l_j$.
\end{itemize}
Now suppose that $\lambda=l_j$ (resp. $\lambda=l_j+1$) for some $j\in\{1,\ldots,r\}$. Then, by definition, $A_j\in \bz$ (resp. $A_j\cup\{n+1\}\in\bz$).
\end{proof}

We are now ready to prove the main result of this section.

\begin{teo}\label{mu hiper}
Let $S\subset\Z^{d}$ be a pointed affine semigroup. Let $w\in\N^n$. Assume that $\Bet(S)=\{b\}$ and $\Bet(S_w)=\{(b,d_w)\}$. Then the M\"obius function of $S_w$ can be computed in terms of data of  $S$. More precisely, for $(z,\lambda)\in S_w$, $\mu_{S_w}(z,\lambda)=0$ whenever $\az=\emptyset$ or $\lambda\notin\{l_j,l_j+1\}_{j=1}^r$. If $\az=\{A_1,\ldots,A_r\}$ and $\lambda\in\{l_j,l_j+1\}_{j=1}^r$, then
\small{
$$\mu_{S_w}(z,\lambda)=\sum_{\lambda=l_j}(-1)^{|A_j|}\binom{k_{A_j}+n-d-1}{k_{A_j}}-\sum_{\lambda=l_j+1}(-1)^{|A_j|}\binom{k_{A_j}+n-d-1}{k_{A_j}}.$$
}
\end{teo}
\begin{proof}
The fact that $S$ is pointed implies that $S_w$ is pointed. By Theorem \ref{formula}:
\[\mu_{S_w}(z,\lambda)=
\left\{
\begin{array}{rll}
&0,&\bz=\emptyset\\
&\sum_{B\in\bz}(-1)^{|B|}\binom{k_B+n-d-1}{k_B}, &\bz\neq\emptyset.\tag{*}\\
\end{array}
\right.
\]
Lemmas \ref{b-a} $(ii)$ and \ref{lambdas}, together with (*) imply that $\mu_{S_w}(z,\lambda)=0$ whenever $\az=\emptyset$ or $\lambda\notin\{l_j,l_j+1\}_{j=1}^r$ (in both cases $\bz=\emptyset$). 

Let $\az=\{A_1,\ldots,A_r\}$, $r\geq1$. 
%Assume that $\bz\neq\emptyset$. 
Let $j_0\in\{1,\ldots,r\}$ be such that $\lambda=l_{j_0}$ or $\lambda=l_{j_0}+1$. Let $\az^0:=\{A_j|\lambda=l_j\}$ and $\az^1:=\{A_j|\lambda=l_j+1\}$. The proof of Lemma \ref{lambdas} shows that $\az^0=\{B\in\bz|n+1\notin B\}$ and $\az^1=\{B\setminus\{n+1\}|B\in\bz, n+1\in B\}$. Using these facts and Lemma \ref{b-a}, (*) implies:
\small{
\begin{align}
\mu_{S_w}(z,\lambda)&=\sum_{\substack{B\in\bz \\ n+1\notin B}}(-1)^{|B|}\binom{k_B+n-d-1}{k_B}+\sum_{\substack{B\in\bz \\ n+1\in B}}(-1)^{|B|}\binom{k_B+n-d-1}{k_B}\notag\\
&=\sum_{A_j\in\az^0}(-1)^{|A_j|}\binom{k_{A_j}+n-d-1}{k_{A_j}}+\sum_{A_j\in\az^1}(-1)^{|A_j|+1}\binom{k_{A_j}+n-d-1}{k_{A_j}}\notag\\
&=\sum_{A_j\in\az^0}(-1)^{|A_j|}\binom{k_{A_j}+n-d-1}{k_{A_j}}-\sum_{A_j\in\az^1}(-1)^{|A_j|}\binom{k_{A_j}+n-d-1}{k_{A_j}}.\notag
\end{align}
}
\end{proof}

\begin{exam}
Let $S$ and $w$ be a numerical semigroup and a weight vector, respectively, such that $|\Bet(S)|=|\Bet(S_w)|=1$ (see Example \ref{un betti w}). It is known that $|\az|\leq1$ for all $z\in S$ \cite[Proof of Corollary 4.2]{CGMR}. By Theorem \ref{formula},
\[\mu_{S}(z)=
\left\{
\begin{array}{rll}
&0,&\az=\emptyset,\\
&(-1)^{|A|}\binom{k_{A}+n-2}{k_{A}}, &\az=\{A\}.\\
\end{array}
\right.
\]
Comparing this formula with Theorem \ref{mu hiper} we obtain
\[\mu_{S_w}(z,\lambda)=
\left\{
\begin{array}{rll}
&0,&\az=\emptyset,\\
&\mu_S(z), &\az=\{A\},\lambda=l,\\
&-\mu_S(z), &\az=\{A\},\lambda=l+1,\\
&0, &\az=\{A\},\lambda\notin\{l,l+1\}.
\end{array}
\right.
\]
\end{exam}

\begin{rem}
The formulas of the previous example are also valid for any semigroup $S$ satisfying the conditions of Theorem \ref{mu hiper} with the extra assumption $|\az|\leq1$ for all $z\in S$.
\end{rem}

\section*{Acknowledgements}

We would like to thank the referees for the very careful reading and many valuable comments that greatly improved the presentation of this paper.

\vspace{.5cm}
\noindent{\footnotesize \textsc {Hern\'an de Alba Casillas, Universidad Aut\'onoma de Zacatecas - CONACYT, Calzada Solidaridad y Paseo de la Bufa, Zacatecas, Zac. 98000, Mexico,} \\
halba@uaz.edu.mx}\\
{\footnotesize \textsc {Daniel Duarte, Centro de Ciencias Matem\'aticas, UNAM, Campus Morelia, Antigua Carretera a P\'atzcuaro 8701, Col. Ex-Hacienda San Jos\'e de la Huerta, 58089, Morelia, Michoac\'an, Mexico,} \\
adduarte@matmor.unam.mx}\\
{\footnotesize \textsc {Ra\'ul Vargas Antuna, Centro de Ciencias Matem\'aticas, UNAM, Campus Morelia, Antigua Carretera a P\'atzcuaro 8701, Col. Ex-Hacienda San Jos\'e de la Huerta, 58089, Morelia, Michoac\'an, Mexico,} \\
raul.vargas@cimat.mx}

\end{document}